\documentclass[12pt]{amsart}
\usepackage{comment}
\usepackage{amsmath,amssymb,amsthm}
\usepackage{tikz}
\usepackage{pgfplots}
\usetikzlibrary{shapes.geometric, lindenmayersystems, decorations.pathmorphing}
\usetikzlibrary{calc}
\pgfplotsset{compat=newest}
\makeatletter
\@namedef{subjclassname@2020}{%
	\textup{2020} Mathematics Subject Classification}
\makeatother

\UseRawInputEncoding

\newtheorem{theorem}{Theorem}[section]
\newtheorem{lemma}[theorem]{Lemma}
\newtheorem{proposition}[theorem]{Proposition}

\newtheorem{question}[theorem]{Question}
\newtheorem{conjecture}[theorem]{Conjecture}

\theoremstyle{definition}
\newtheorem{definition}[theorem]{Definition}
\newtheorem{example}[theorem]{Example}

\theoremstyle{remark}
\newtheorem{remark}[theorem]{Remark}
\numberwithin{equation}{section}

\def\C{\mathbb C}

\def\D{\mathbb D}

\def\R{\mathbb R}

\begin{document}

\dedicatory{In memory of Prof. Marek Jarnicki}

\numberwithin{equation}{section}
	
	\title[Open questions and conjectures in convex domains]
	{Some open questions and conjectures about visibility and iteration in bounded convex domains in $\C^N$}

	\author{Filippo Bracci}
	\author{Ahmed Yekta \"Okten}  
	
	\address{F.Bracci, A. Y. \"Okten:  Dipartimento di Matematica, Universit\`a Di Roma ``Tor Vergata''€, Via Della Ricerca Scientifica 1, 00133 Roma, Italy}

	\email{fbracci@mat.uniroma2.it}

	\email{okten@mat.uniroma2.it}

	\thanks{Partially supported by the MIUR Excellence Department Project 2023-2027 MatMod@Tov awarded to the
		Department of Mathematics, University of Rome Tor Vergata, by PRIN Real and Complex Manifolds: Topology,
		Geometry and holomorphic dynamics n.2017JZ2SW5 and by GNSAGA of INdAM}

	\begin{abstract}  
In this note, we propose some open problems and questions about bounded convex domains in $\C^N$, specifically about visibility and iteration theory.	
 \end{abstract}

	
	
	\maketitle
	
\section{Introduction}

Iteration theory in complex domains is a very classical subject of study, dating back to the pioneering papers of Denjoy \cite{Den} and Wolff \cite{Wol1, Wol2}. (See also \cite{AbateTaut, BCDM, AbateNewbook} for further references.) Even for holomorphic self-maps of the unit disc, many open problems and new interesting results emerge annually. Moreover, in higher dimensions, even more open questions remain. 

In this note, we present open questions, a few conjectures, and some new results about iteration theory in bounded convex domains. The reason for staying within this class is that it exhibits peculiar properties, primarily discovered by L. Lempert \cite{Lem} (see also \cite{JP, AbateTaut}), which provide an incredible array of examples and results. 

In the past ten years, significant efforts have been dedicated to applying Gromov hyperbolicity theory in several complex variables, leading to the discovery of new and intriguing results. A. Karlsson \cite{Kar} proved that the Denjoy-Wolff Theorem holds in Gromov hyperbolic spaces for the Gromov topology. G. Bharali and A. Zimmer \cite{BZ, BZ2}, and later N. Nikolov, P. Thomas and the first author \cite{BNT}, identified a distinctive feature of Gromov hyperbolic spaces known as {\sl visibility},  which is related to the ``bending inside'' of geodesics. Visibility has far-reaching implications. For instance, G. Bharali and A. Maitra \cite{BM} showed that visibility, even in the absence of Gromov hyperbolicity, implies the Denjoy-Wolff Theorem. However, as proven in \cite{BB}, visibility is a sufficient condition for the Denjoy-Wolff theorem, but it is not  necessary in general.  

For bounded convex domains in $\C^N$, the known results about visibility and the Denjoy-Wolff theorem suggest that these two properties are equivalent. This note aims to provide a tour of the relationship between visibility and the Denjoy-Wolff theorem, reformulating the notions of visibility, Bharali-Maitra's theorem, and key questions and conjectures that could lead to proving their equivalence for bounded convex domains (regardless of any regularity conditions).

\medbreak

We wish to thank the anonymous referee for several comments which improved the readability of the manuscript.

\section{Notations}

The Kobayashi distance of a bounded convex domain $\Omega\subset\C^N$ is denoted by $k_\Omega$, while the Kobayashi-Royden infinitesimal (pseudo)metric by $\kappa_\Omega$. We refer to \cite{JP, AbateTaut, Lem} for properties and details. In particular, $k_\Omega$ is a complete distance on $\Omega$ which defines the Euclidean topology of $\Omega$ and it is contracted by holomorphic maps.

A {\sl complex geodesic} $\varphi: \D\to \Omega$ is a holomorphic map such that it is an isometry between $k_\D$ and $k_\Omega$. For every $z,w\in \Omega$ there exists a complex geodesic such that $z,w\in \varphi(\D)$. Moreover, given a complex geodesic $\varphi:\D\to\Omega$ there exists a holomorphic map $\rho:\Omega\to\D$ such that $\rho \circ \varphi={\sf id}_\D$. Such a map is called a {\sl left-inverse} of $\varphi$. 

A {\sl geodesic segment} $\gamma:[0,1]\to \Omega$ is an absolutely continuous curve such that for all $0\leq s\leq t\leq 1$
\[
\int_s^t \kappa_\Omega(\gamma(r); \gamma'(r))dr=k_\Omega(\gamma(s), \gamma(t)).
\]
The interval $[0,1]$ can be replaced by any segment in $\mathbb{R}$. However, in this paper, unless otherwise specified, when referring to a geodesic segment, we always assume it is defined in the interval $[0,1]$.

A {\sl geodesic ray} $\gamma:[0,+\infty)\to \Omega$ is an absolutely continuous curve such that $\gamma|_{[0,t]}$ is a geodesic segment for all $t\in [0,+\infty)$ and $\lim_{t\to+\infty}k_\Omega(\gamma(0), \gamma(t))=+\infty$.  The choice of interval $[0,+\infty)$ can be replaced by any interval $[a, b)$ with $a<b$, but unless otherwise specified, when referring to a geodesic ray, we always assume its interval of definition is $[0,+\infty)$.

Note that if $\varphi:\D\to \Omega$ is a complex geodesic, then $\varphi$ maps geodesic segments or rays of $\D$ onto geodesic segments or rays of $\Omega$. Therefore, for every pair of distinct points $z,w\in\Omega$, there exists a geodesic segment joining $z$ and $w$. This follows also directly from the Hopf-Rinow theorem since $k_\Omega$ is complete.

\section{Visibility and complex geodesics}

G. Bharali and A. Zimmer introduced the concept of visibility for, non-necessarily complete, hyperbolic domains in \cite{BZ} involving almost-geodesics as defined by them. For bounded convex domains, since any two points are joined by a geodesic segment, the paper \cite{BNT} introduced a simpler notion of visibility. Note that by the works in \cite{BNT} and \cite{CMS}, the visibility notion introduced in \cite{BNT} is equivalent to a weaker form of the notion introduced by Bharali and Zimmer. It remains an open question whether, in general, those two definitions are equivalent.

Since we are working with bounded convex domains, we adopt a refined version of the definition given in \cite{BNT}.

\begin{definition}
Let $\Omega\subset\C^N$ be a  bounded convex domain. Let $p, q\in \partial \Omega$, $p\neq q$. We say that $p$ and $q$ are {\sl essentially visible} if there exists a compact set $K\subset\subset \Omega$ such that for every sequences $\{z_k\}, \{w_k\}\subset\Omega$ so that $\{z_k\}$ converges to $p$ and $\{w_k\}$ converges to $q$ there exists a sequence of geodesic segments $\{\gamma_k\}$ such that $\gamma_k(0)=z_k$, $\gamma_k(1)=w_k$ and $\gamma_k(s_k)\in K$ for some $s_k\in (0,1)$ and for all $k$.

We say that $p$ and $q$ are {\sl strongly visible} if there exists a compact set $K\subset\subset \Omega$ so that for every sequence of geodesic segments $\{\gamma_k\}$ such that $\{\gamma_k(0)\}$ converges to $p$ and  $\{\gamma_k(1)\}$ converges to $q$ there is $s_k\in (0,1)$ such that $\gamma_k(s_k)\in K$  for all $k$ sufficiently large. 

We say that $q$ is {\sl non-visible} with respect to $p$ (or, equivalently, $p$ is {\sl non-visible} with respect to $q$) if $p$ and $q$ are not essentially visible.

If every $p, q\in \partial \Omega$, $p\neq q$ are essentially visible, we say that $\Omega$ is {\sl essentially visible}. While, if  every $p, q\in \partial \Omega$, $p\neq q$ are strongly visible, we say that $\Omega$ is {\sl  visible}. 
\end{definition}

Clearly, strongly visible pairs of points are essentially visible. The following example shows that the converse does not hold in general. 
 
 \begin{example}\label{example:bidisc}
 	Let $\Omega:=\mathbb D \times \mathbb D$ be the bidisc in $\mathbb C^2$, let $r_k \in (0,1)$ be a sequence increasing to one and set $z_k:=(-r_k,0),w_k:=(r_k,0)$. Since $k_\Omega((z_1,z_2),(w_1,w_2))=\max\{k_{\mathbb{D}}(z_1,w_1),k_{\mathbb{D}}(z_2,w_2)\}$ and $\kappa_\Omega((z_1,z_2);(v_1,v_2))=\max\{\kappa_{\mathbb{D}}(z_1;v_1),\kappa_{\mathbb{D}}(z_2;v_2)\}$ it follows that the curves $\gamma_k:[0,1]\to\Omega$ defined by $$
 	\gamma_k(t) = 
 	\begin{cases}
 		\left(\dfrac{r_k(2t-1)}{1-2r^2_k t},2tr_k\right) & \text{if} \:\:\:t\in[0,1/2], \\
 		\left(\dfrac{r_k(2t-1)}{1+r^2_k(2t-2)},(2-2t)r_k\right) &  \text{otherwise}  
 	\end{cases}
 	$$ are geodesic segments joining $z_k$ to $w_k$. As $r_k$ tends to one $\gamma_k([0,1])$ eventually avoids any compact set in $\Omega$. Consequently, it follows that the pair $(-1,0)$ and $(1,0)$ in $\partial\Omega$ are not strongly visible. 
 	
 	Denote by $\pi_1:\mathbb C^2 \to \mathbb C$, $\pi_2:\mathbb C^2 \to \mathbb C$ projections to first and second coordinates respectively. Observe that if $z_k,w_k \in \Omega$ tend to $(-1,0)$ and $(1,0)$ respectively, and $\gamma_k:[0,1]\to\Omega$ are geodesic segments such that $\gamma_k(0)=z_k$ and $\gamma_k(1)=w_k$, the visibility of the unit disc shows that $\pi_2\circ\gamma_k([0,1])$ meet some compact set in $\mathbb D$. As $\pi_1(z_k),\pi_1(w_k)$ are compactly contained in $\mathbb D$, it is not difficult to conclude that $z_k$ and $w_k$ can be joined by geodesic segments passing through a fixed compact set in $\Omega$. We conclude that $(-1,0)$ and $(1,0)$ are essentially visible.
 \end{example}

The situation in the bidisc leads us to ask the following question.

\begin{question}
	Does there exist a bounded convex domain $\Omega\subset\C^N$ with no non-trivial analytic discs in the boundary which is essentially visible but not visible? Does there exist a smoothly bounded convex domain $\Omega\subset\C^N$ which is essentially visible but not visible? 
\end{question}

Bounded convex domains which are Gromov hyperbolic with respect to the Kobayashi distance are visible \cite{BGZ}. In particular, smoothly bounded convex finite type domains are visible  \cite{Zi1, Zi2}. However visibility is  weaker than Gromov hyperbolicity: indeed, smoothly bounded convex domains for which all, but a finite number of points, are of finite type  are visible \cite[Theorem 1.1]{BNT} but they are not Gromov hyperbolic with respect to the Kobayashi distance. In fact, there are examples of bounded convex domains containing a complex tangential line segment of infinite type points that  are visible \cite[Section~5]{BNT}.

Therefore it seems very difficult to find a sharp characterisation of visibility for convex domains in terms of the Euclidean geometry of their boundaries. Nevertheless, we can attempt to understand the possible cluster sets of ``non-visible' geodesic segments. To be precise, we need a definition.

\begin{definition}
Let $\Omega\subset\C^N$ be a  bounded convex domain. Let $\{\gamma_k\}$ be a sequence of geodesic segments. We say that $\{\gamma_k\}$ is a sequence of {\sl non-visible geodesic segments} if 
\begin{enumerate}
\item for every compact set $K\subset\subset \Omega$ there exists $k_0$ such that $\gamma_k([0,1])\cap K=\emptyset$ for all $k\geq k_0$,
\item there exist $p, q\in \partial\Omega$, $p\neq q$ such that $\lim_{k\to \infty}\gamma_k(0)=p$ and $\lim_{k\to \infty}\gamma_k(1)=q$.
\end{enumerate}
\end{definition}

We have the following

\begin{conjecture}\label{Conj1}
Let $\Omega\subset\C^N$ be a  bounded convex domain. Let $\{\gamma_k\}$ be a sequence of  non-visible geodesic segments and let 
\[
\Gamma(\{\gamma_k\}):=\{p\in\partial\Omega: \exists \{t_{k_n}\}\subset [0,1]: \lim_{n\to \infty}\gamma_{k_n}(t_{k_n})=p\}.
\] 
If $p, q\in \Gamma(\{\gamma_k\})$ and $p\neq q$ then
\begin{itemize}
\item either $\{tp+(1-t)q, 0<t<1\}\subset\Omega$,
\item or, $\{tp+(1-t)q, 0\leq t\leq 1\}\subset\partial\Omega$ and $(\C(p-q)+q)\cap \Omega=\emptyset$.
\end{itemize}
\end{conjecture}

The second condition in the previous Conjecture means that the real segment $\{tp+(1-t)q, 0\leq t\leq 1\}$ belongs to a complex supporting hyperplane of $\Omega$, or, in other terms, it is ``complex-tangential''.

The (Euclidean) geometry of the bidisc suggest that the above conjecture trivially holds for the bidisc. Furthermore, in \cite[Section 5]{BNT} and \cite[Section 5]{Okt} there are examples of smooth convex non-visible domains with no analytic discs in the boundary. For such domains, the limit set $\Gamma(\{\gamma_k\})$ of non-visible geodesic segments $\gamma_k$ lie in a real line segment contained in the intersection of a complex hyperplane with the boundary.

Actually, the examples mentioned above are not coincidental. As a consequence of \cite[Theorem 4.1]{Zi0} if $\Omega\subset\mathbb C^N$ is a bounded convex domain with $\mathcal C^{1+\varepsilon}$-smooth boundary, it is shown in \cite[Theorem 11]{Okt} that if $\gamma_k:[0,1]\to\Omega$ is a sequence of geodesic segments such that $\gamma_k(0)$ tends to $p\in\partial\Omega$ and $\gamma_k([0,1])$ tends to $\partial\Omega$, then $\gamma_k([0,1])$ tends to the affine complex tangent hyperplane to $\partial\Omega$ at $p$, that is $\Gamma(\{\gamma_k\})\subset T^{\mathbb C}_p \partial \Omega \cap \partial \Omega$. Consequently the above conjecture is true (and only the second case can occur) under the hypothesis of $\mathcal C^{1+\epsilon}$-smoothness of $\partial\Omega$.

If Conjecture~\ref{Conj1} is true, then every bounded convex domain $\Omega$ with a strictly $\C$-linearly convex boundary (that is, every complex affine line $L$ which does not intersect $\Omega$  has the property that $L\cap \partial \Omega$ contains at most one point) is visible. Actually, it is known that  bounded (local) strictly $\C$-linearly convex domains with Dini-smooth boundaries  are visible \cite[Theorem 1.3]{BNT}.

An interesting characterization of visibility in terms of complex geodesics is provided by the following result of N. Nikolov (\cite[Prop.~5]{Nik}). Let $\Omega\subset\C^N$ be a  bounded convex domain. Let $\mathcal F_0$ be the family of all complex geodesics $\varphi$ of $\Omega$ parametrized so that $\delta_\Omega(\varphi(0))=\max_{\zeta\in\D}\delta_\Omega(\varphi(\zeta))$, where, as usual, $\delta_\Omega(z)$ is the distance of $z\in \Omega$ from $\partial \Omega$.

\begin{theorem}[Nikolov]
Let $\Omega\subset\C^N$ be a  bounded convex domain. Then $\Omega$ is visible if and only if
\begin{enumerate}
\item $\partial\Omega$ does not contain non-trivial analytic discs,
\item $\mathcal F_0$ is uniformly equicontinuous (as family of functions from $\D$ endowed with the Euclidean topology to $\C^N$ endowed with the Euclidean topology).
\end{enumerate}
\end{theorem}

Is it not clear whether condition (2) in the previous theorem implies visibility by itself, but presumably it does. 

Note that in particular, the previous theorem implies that if $\Omega$ is visible then every complex geodesic $\varphi:\D\to \Omega$ extend continuously up to $\overline{\D}$. 

In \cite{BNT} it has been proved that if $D\subsetneq \C$ is a bounded simply connected domain then $D$ is visible if and only if $\partial D$ is locally connected. In particular, by Carath\'eodory's theorem (see, {\sl e.g.}, \cite{BCDM}) for a simply connected domain a Riemann map extends continuously up to the boundary if and only if the domain is visible. It seems interesting to answer the following question:

\begin{question}\label{Q:continuity}
Does there exist a non-visible bounded convex domain $\Omega$ such that every complex geodesic $\varphi:\D\to\Omega$ extends continuously up to $\partial\D$? In other words, is the continuous extension of all complex geodesics up to the boundary  sufficient for visibility?

\end{question}

Also, Nikolov's Theorem hints to the following notion of visibility:

\begin{definition}
Let $\Omega\subset\C^N$ be a  bounded convex domain. We say that $\Omega$ is {\sl complex visible} if for every $p, q\in \partial\Omega$ such that $p\neq q$ there exist a compact set $K\subset\subset \Omega$ such that for every sequences $\{z_k\}, \{w_k\}\subset\Omega$ so that $\{z_k\}$ converges to $p$ and $\{w_k\}$ converges to $q$ every sequence of complex geodesics $\{\varphi_k\}$ such that $z_k, w_k\in \varphi_k(\D)$ has the property that $\varphi_k(\D)\cap K\neq\emptyset$ for large enough $k$.
\end{definition}

It is clear that visibility implies complex visibility.

\begin{conjecture}
Let $\Omega\subset\C^N$ be a  bounded convex domain. Then $\Omega$ is visible if and only if it is complex visible.
\end{conjecture}

Let $\Omega\subset\C^N$ be a  bounded convex domain. Let $z_0\in \Omega$ and let $\mathcal G_{z_0}$ be the family of all geodesic segments and geodesic rays starting from $z_0$, parameterized in  hyperbolic arc-length\footnote{Here we say that a geodesic segment $\gamma:[0,R]\to \Omega$, $R>0$, starts from $z_0$ and it is parametrized in hyperbolic arc-length if $\gamma(0)=z_0$ and if $\kappa_\Omega(\gamma(s); \gamma'(s))=1$ for almost all $s\in [0,R]$. Similarly for a geodesic ray, taking $R=+\infty$.}.

From \cite[Lemma~3.1]{Br} it follows that if $\Omega$ is a visible bounded convex domain then $\mathcal G_{z_0}$ is uniformly equicontinuous. In particular, all geodesic rays {\sl land} (that is, if $\gamma:[0,+\infty)\to\Omega$ is a geodesic ray in $\Omega$, then $\lim_{t\to\infty}\gamma(t)$ exists). 

\begin{question}
Does there exist a bounded convex domain $\Omega$ which is not visible and such that all geodesic rays land?
\end{question}

It is easy to construct examples of (not convex) simply connected domain in $\C$ which are not visible and for which all geodesic rays land (see, {\sl e.g.}, \cite{BB}): in fact, any bounded simply connected domain for which the principal part of any prime end is a singleton but there is a prime end whose impression is not a single point provides such an example. We would expect the answer to the previous question to be negative, as otherwise the landing of all geodesic rays would determine visibility. However, we do not know of any counterexample.

However, Nikolov's Theorem and the results in \cite{Br} make plausible the following:

\begin{conjecture}
Let $\Omega\subset\C^N$ be a  bounded convex domain. Then $\Omega$ is visible if and only if $\mathcal G_{z_0}$ is uniformly equicontinuous (as family of functions from $[0,+\infty)$ endowed with the Euclidean topology to $\C^N$ endowed with the Euclidean topology).
\end{conjecture}

\section{The Denjoy-Wolff property}

Let $\Omega\subset\C^N$ be a bounded convex domain. Let $F:\Omega\to \Omega$ be a holomorphic map.  The {\sl target set of $F$} is the set 
\[
T(F):=\{z\in \partial{\Omega}: \exists z_0 \in \Omega,  \{n_k\}\subset \mathbb N  \hbox{ such that }  n_k\to \infty \hbox{ and } F^{\circ n_k}(z_0) \to z \:\: \text{as} \:\: k \to \infty\}.
\]

Abate (see, \cite[Theorem~2.4.20]{AbateTaut}) proved that    $T(F)=\emptyset$ if and only if $F$ has fixed points in $\Omega$.

\begin{definition}\label{def:wolffdenjoyproperty}
We say that  $\Omega$ has  \emph{the Denjoy-Wolff property} if for any holomorphic self-map $F$ of $\Omega$ without fixed points,  there exists a point $p\in\partial\Omega$ such that  $T(F)=\{p\}$. 
\end{definition}

It is known since Herv\'e \cite{Her2}  that the unit ball $\mathbb B^N$ of $\mathbb C^N$ verifies the Denjoy-Wolff property. Abate in \cite{Aba, AbateTaut} extended the result of Herv\'e to strongly convex domains with $\mathcal C^2$-smooth boundaries. Later Budzy\'nska \cite{Bud} (see also \cite{AR}) removed the boundary regularity assumptions and extended this result to strictly convex domains, while in \cite{BGZ} the same result has been proved for bounded convex domains which are Gromov hyperbolic with respect to the Kobayashi distance, in particular, this happens for smoothly bounded finite type convex domains (see \cite{Zi1, Zi2}). See also \cite{BKR, CR} and references therein for relevant results about the Denjoy-Wolff theorem in complex Banach spaces and in  symmetric domains.

If $\partial \Omega$ does not have non-trivial analytic discs, the target set can be determined by the cluster set of a unique point of $\Omega$, namely

\begin{lemma}\label{Lem:TF-unique}
Let $\Omega\subset\C^N$ be a convex bounded domain. Assume $\partial \Omega$ does not have non-trivial analytic discs. Let $F:\Omega\to \Omega$ be a holomorphic self-map without fixed points. Let $z_0\in \Omega$. Then 
\[
T(F)=T(F, z_0):=\{q\in \partial \Omega: \exists  \{n_k\}\subset \mathbb N \:\: \text{such that} \:\: n_k\to \infty, F^{\circ n_k}(z_0) \to z \:\: \text{as} \:\: k \to \infty\}.
\]
\end{lemma}
\begin{proof}
Let $\{n_k\}$ be a sequence converging to $\infty$ and assume $\{F^{\circ n_k}(z_0)\}$ converges to some point $p\in \partial \Omega$. Let $z_1\in \Omega$. Then for all $k$.
\[
k_\Omega(F^{\circ n_k}(z_0), F^{\circ n_k}(z_1))\leq k_\Omega(z_0, z_1)<+\infty.
\]
Since $k_\Omega$ is complete, it follows that $\{F^{\circ n_k}(z_1)\}$ converges to $\partial \Omega$. Since by hypothesis $\partial\Omega$ does not have non-trivial analytic discs, by D'Addezio's lemma \cite[Lemma~A.2]{BG}, $\{F^{\circ n_k}(z_1)\}$ converges to $p$.
\end{proof}

We borrow the following definition from A. Karlsson \cite{Kar}:

\begin{definition}
Let $\Omega\subset\C^N$ be a convex bounded domain. Let $F:\Omega\to\Omega$ be a holomorphic self-map without fixed points. Let $z_0\in \Omega$. Let $\{n_k\}\subset\mathbb N$ be a sequence converging to $\infty$. We say that $\{n_k\}$ is a {\sl record sequence} for $z_0$ provided for all $m\leq n_k$ and for all $k$ we have
\[
k_\Omega(F^{\circ m}(z_0), z_0)\leq k_\Omega(F^{\circ n_k}(z_0), z_0).
\]
\end{definition}

Record sequences are the basic tool to produce Julia's version of the Denjoy-Wolff theorem (see, {\sl e.g.}, \cite{Aba, AbateTaut}). 

\begin{remark}
Let $\Omega$ be a bounded convex domain. Let $F$ be a holomorphic self-map of $\Omega$ without fixed points, and let $z_0\in \Omega$. For every $k\in\mathbb N$ we can define $n_k\in \{1,\ldots, k\}$ to be such that $k_\Omega(F^{\circ n_k}(z_0), z_0)$ is the maximum among $\{k_\Omega(F^{\circ m}(z_0), z_0)\}_{m=1,\ldots, k}$. Since $k_\Omega$ is  complete, we have $\lim_{n\to\infty}k_\Omega(F^{\circ n}(z_0), z_0)=\infty$. Hence, $\{n_k\}$ is converging to $+\infty$ and it is thus a record sequence for $z_0$. 
\end{remark}

G. Bharali and A. Maitra in \cite{BM} proved that if a bounded (not necessarily convex) domain is visible, then it has the Denjoy-Wolff property. Their argument can easily be adapted to prove the following  result:

\begin{theorem}
Let $\Omega\subset \C^N$ be a bounded convex domain. Assume $\partial\Omega$ does not contain non-trivial analytic discs. Let $F:\Omega\to\Omega$ be a holomorphic self-map without fixed points. Let $\{n_k\}$ be a record sequence for $z_0\in \Omega$ and assume that $\{F^{\circ n_k}(z_0)\}$ converges to some point $p\in \partial\Omega$. Then every $q\in T(F)\setminus\{p\}$ is non-visible with respect to $p$. In particular, if $\Omega$ is complex visible or essentially visible then it enjoys the Denjoy-Wolff property.
\end{theorem}
\begin{proof}
The proof is essentially the one provided in \cite{BM}, but we give a version here for the sake of completeness.
 
By Lemma~\ref{Lem:TF-unique}, every point $q\in T(F)$ is the limit of $\{F^{\circ m_k}(z_0)\}$ for some sequence $\{m_k\}$ converging to $\infty$. Let $q\in T(F)\setminus\{p\}$, and assume by contradiction that $q$ is essentially visible with respect to $p$. Furthermore, assume $q=\lim_{k\to \infty}F^{\circ m_k}(z_0)$ for some sequence $\{m_k\}$ converging to $\infty$. Up to passing to a subsequence if necessary, we can assume that $m_k\leq n_k$ for all $k$. 

Since we are assuming by contradiction that $p$ and $q$ are essentially visible, there is a compact set $K\subset\subset\Omega$ and geodesic segments $\gamma_k:[0,1]\to \Omega$  such that $\gamma_k(0)=F^{\circ n_k}(z_0)$, $\gamma_k(1)=F^{\circ m_k}(z_0)$ and $\gamma_k(s_k)\in K$ for some $s_k\in (0,1)$ and for all $k\in \mathbb N$. In particular, there exists $C>0$ such that $k_\Omega(z_0, \gamma_k(s_k))\leq C$ for all $k$. Hence, for all $k$
\[
\begin{split}
k_\Omega(F^{\circ m_k}(z_0), \gamma_k(s_k))&=k_\Omega(F^{\circ m_k}(z_0), F^{\circ n_k}(z_0))-k_\Omega(\gamma_k(s_k), F^{\circ n_k}(z_0))\\&\leq k_\Omega(F^{\circ (n_k-m_k)}(z_0), z_0)-k_\Omega(\gamma_k(s_k), F^{\circ n_k}(z_0))\\&\leq k_\Omega(F^{\circ n_k}(z_0), z_0)-k_\Omega(\gamma_k(s_k), F^{\circ n_k}(z_0))\\&\leq k_\Omega(z_0, \gamma_k(s_k))\leq C.
\end{split}
\]
Since $k_\Omega$ is complete, this implies that $\{F^{\circ m_k}(z_0)\}$ is relatively compact in $\Omega$, a contradiction.
\end{proof}

Record sequences can also be used to create ``sequential horospheres'' which are invariant for a given holomorphic self-map. 
\begin{definition}
Let $\Omega\subset \C^N$ be a bounded convex domain. Let $\{z_n\}\subset\Omega$ be a sequence converging to some $p\in \partial \Omega$. Let $R>0$. We let
\[
E_{z_0}(\{z_n\}, R):=\{z\in \Omega: \limsup_{n\to\infty}[k_\Omega(z, z_n)-k_\Omega(z_0, z_n)]<\frac{1}{2}\log R\}.
\]
\end{definition}
The set $E_{z_0}(\{z_n\}, R)$, if not empty, is convex (see \cite[Proposition~6.1]{BGh}) and it is a ``sequential horosphere'' of radius $R$. We have the following (see \cite{AR, Bud}):

\begin{proposition}
Let $\Omega\subset \C^N$ be a bounded convex domain. Let $F:\Omega\to\Omega$ be holomorphic and without fixed points. Let $\{n_k\}$ be a record sequence for $z_0\in \Omega$ and assume that $\{F^{\circ n_k}(z_0)\}$ converges to some point $p\in \partial\Omega$. Then for every $n\in \mathbb N$
\[
F^{\circ n}(z_0)\in \overline{E_{z_0}(\{F^{\circ n_k}(z_0)\}, 1)}.
\]
In particular, if $\partial \Omega$ does not have non-trivial analytic discs then
\[
T(F)\subseteq \partial \Omega\cap \partial E_{z_0}(\{F^{\circ n_k}(z_0)\}, 1).
\]
\end{proposition}
\begin{proof}
For all $n\in \mathbb N$, if $n_k>n$ (that is, for $k$ sufficiently large)
\[
\begin{split}
k_\Omega(F^{\circ n}(z_0), F^{\circ n_k}(z_0))-k_\Omega(z_0, F^{\circ n_k}(z_0))&\leq k_\Omega(F^{\circ (n_k-n)}(z_0), z_0)-k_\Omega(z_0, F^{\circ n_k}(z_0))\\
&\leq k_\Omega(F^{\circ n_k}(z_0), z_0)-k_\Omega(z_0, F^{\circ n_k}(z_0))\leq 0.
\end{split}
\]
Therefore, $F^{\circ m}(z_0)\in \overline{E_{z_0}(\{F^{\circ n_k}(z_0)\}, 1)}$. The last statement follows at once from Lemma~\ref{Lem:TF-unique}.
\end{proof}

The previous proposition implies that if the closure of any sequential horosphere of a bounded convex domain without non-trivial analytic discs on its boundary touches the boundary at exactly one point, then the domain possesses the Denjoy-Wolff property.

In \cite{BB}, it has been proven that for simply connected domains in $\C$ (distinct from $\C$)   the Denjoy-Wolff property is actually equivalent to the fact that the closure of any sequential horosphere touches the boundary of the domain at exactly one point.

\begin{conjecture}
Let $\Omega\subset \C^N$ be a bounded convex domain. Assume $\partial\Omega$ does not contain non-trivial analytic discs. Then $\Omega$ has the Denjoy-Wolff property if and only if the closure of any sequential horosphere  touches $\partial\Omega$ at just one point.
\end{conjecture}

In the recent paper \cite{BO}, the authors constructed examples of $C^\infty$-smooth bounded convex domains without non-trivial analytic discs on the boundary for which the Denjoy-Wolff property does not hold. Those domains are in fact strongly convex outside a real segment $I$ sitting on the boundary which is made of points of infinite type. This resembles very much the situations which brought us to formulate Conjecture~\ref{Conj1}, therefore, in \cite{BO} we formulated the following:

\begin{conjecture}
Let $\Omega\subset \C^N$ be a bounded convex domain. Then $\Omega$ is visible if and only if it enjoys the Denjoy-Wolff property.
\end{conjecture}

Note that, if one could answer Question~\ref{Q:continuity} by showing that  non-visible bounded convex domains always have a complex geodesic $\varphi:\D\to\Omega$ which does not have non-tangential limit (or more generally $H$-limits in the sense of \cite{BB}) at some point $\sigma\in\partial\D$ then one could easily construct a holomorphic self-map $F$ of $\Omega$ without fixed points and such that its target set is not a singleton. Indeed, if this is the case, just take a parabolic automorphism $f$ of $\D$ with Denjoy-Wolff point at $\sigma$. Let  $\rho:\Omega\to \D$ be a left-inverse of $\varphi$, then the holomorphic self-map $F:= \varphi\circ f\circ  \rho$ of $\Omega$ has no fixed points in $\Omega$ and its target set contains the non-tangential (or the $H$-)cluster set of $\varphi$ at $\sigma$.

\section{The polydisc and product domains}

M. Herv\'e \cite{Her} proved the following result: if $F:\D^2\to \D^2$ is a holomorphic self-map without fixed points then either $T(F)\subseteq \overline{\D}\times \{e^{i\theta}\}$ or $T(F)\subseteq \{e^{i\theta}\}\times\overline{\D}$ for some $\theta\in \R$. For the polydisc $\D^N$, $N>2$, M. Abate and J. Raissy \cite[Corollary~5]{AR} proved that if $F$ is a holomorphic self-map of $\D^N$ without fixed points then
\[
T(F)\subseteq \bigcup_{j=1}^N \overline{\D}\times\ldots \times C_j(\xi)\times \ldots \overline{\D},
\]
for some $\xi\in \partial\D^N$ and
\[
 C_j(\xi)=\begin{cases}  
 \{\xi_j\} \quad \hbox{if } |\xi_j|=1,\\
 \partial\D \quad \hbox{if } |\xi_j|<1.
 \end{cases}
\]
However, there are no examples of holomorphic self-maps of $\D^N$ which have such a behavior. In fact, it is generally believed that the following conjecture is true:
\begin{conjecture}
Let $N\geq 1$ and let $F:\D^N\to \D^N$ be holomorphic and without fixed points. Then there exist $j\in \{1,\ldots, N\}$ and $\sigma\in\partial \D$ such that $\lim_{n\to\infty}\pi_j(F^{\circ n}(z))=\sigma$ for all $z\in \D^N$, where $\pi_j(z)=z_j$ (the projection on the $j$-th component). 
\end{conjecture}

In support of this conjecture we state some results. We start with a lemma which replaces Lemma~\ref{Lem:TF-unique} for the case of the polydisc:

\begin{lemma}\label{Lem:orbits}
Let $N\geq 2$. Let $F:\D^N\to \D^N$ be a holomorphic self-map  without fixed points in $\D^N$. Let $p\in \D^N$ and
\[
T(F,p):=\{\sigma\in \partial(\D^N): \hbox{ there exists $\{n_k\}\subset \mathbb N$ such that } \lim_{k\to\infty}F^{\circ n_k}(p)=\sigma\}.
\]
If $T(F,p)\subseteq \{1\}\times \overline{\D^{N-1}}$ then $T(F)\subseteq \{1\}\times \overline{\D^{N-1}}$.
\end{lemma}
\begin{proof}
Let $\sigma\in \partial(\D^N)$, $q\in \D^N$ and $\{n_k\}\in\mathbb N$ be such that $\lim_{k\to\infty}F^{n_k}(q)=\sigma$. Since $T(F,p)\subseteq \{1\}\times \overline{\D^{N-1}}$ it follows that $\lim_{n\to\infty} \pi(F^{\circ n}(p))=1$ and, in particular, $\lim_{k\to\infty}\pi_1(F^{n_k}(p))=1$.

 Note that $\pi_1:\D^N\to \D$ is holomorphic, hence it contracts the Kobayashi distance. Thus,
\[
K_\D(\pi_1(F^{\circ n_k}(q)), \pi_1(F^{\circ n_k}(p))\leq K_{\D^N}(F^{\circ n_k}(q)), F^{\circ n_k}(p))\leq K_{\D^N}(q, p)<+\infty.
\]
Since  $\{\pi_1(F^{\circ n}(p))\}$ converges to $1$, it follows that also $\{\pi_1(F^{\circ n_k}(q))\}$ converges to $1$, that is, $\pi_1(\sigma)=1$, and we are done.
\end{proof}

Here the number $1$ does not play any special role and it can be clearly substitute with any unimodular number, likewise  $\pi_1$ can be substituted with any $\pi_j$. We also need the following result (see \cite[Theorem~0.7]{Abapoly}) which we state here in a simple form as needed for our aims:

\begin{theorem}[Abate]\label{abatep}
Let $f:\D^N\to \D$ be holomorphic. Let $1\leq \nu\leq N$ and let $\sigma=(1, \ldots, 1, \sigma_{\nu+1}, \ldots, \sigma_N)\in \partial(\D^N)$, with $|\sigma_j|<1$ for all $j=\nu+1, \ldots, N$. Let $\underline{0}=(0,\ldots, 0)$. Suppose  that there exists a sequence $\{\xi_k\}\subset \D^N$ converging to $\sigma$ such that $\{f(\xi_k)\}$ converges to $1$ and there exists $\alpha> 0$ such that
\[
\limsup_{k\to \infty}[K_{\D^N}(\underline{0}, \xi_k)-K_\D(0, f(\xi_k))]\leq \frac{1}{2}\log \alpha
\]
 Then 
\[
\limsup_{r\to 1^-}\frac{|1-f(r \sigma)|}{1-r}\leq \alpha.
\]
\end{theorem}

We can prove the following:

\begin{proposition}
Let $N\geq 2$. Let $F:\D^N\to \D^N$ be a holomorphic self-map  without fixed points in $\D^N$. Then there exist $\sigma\in\partial \D$,  and $j_0\in \{1,\ldots, N\}$ such that for every $m\in \mathbb N$ there exists $q_m\in \partial (\D^N)$ so that for all $\zeta\in \D$
\begin{equation}\label{classJulia}
\frac{|\sigma-\pi_{j_0}(F^{\circ m}(\zeta q_m)) |^2}{1-|\pi_{j_0}(F^{\circ m}(\zeta q_m))|^2}\leq  \frac{|\sigma-\zeta |^2}{1-|\zeta|^2}.
\end{equation}
Moreover, the constant map $\zeta\mapsto \sigma$ is a limit (uniform on compacta) of $\{\pi_{j_0}(F^{\circ m}(\zeta q_m))\}$ and if $g$ is any other limit  of $\{\pi_{j_0}(F^{\circ m}(\zeta q_m))\}$ then
\begin{enumerate}
\item either $g$ is the identity,
\item or, $g$ is a holomorphic self-map of $\D$ without fixed points and with  Denjoy-Wolff point  $\sigma$. 
\end{enumerate}
\end{proposition}
\begin{proof}
Let $\{n_k\}$ be a record sequence for $F$ with respect to $\underline{0}=(0,\ldots, 0)$.

Up to extracting subsequences from $\{n_k\}$, we can assume that $\{F^{\circ n_k}(\underline{0}))\}$ converges to some point $q\in\partial (\D^N)$. Let $\pi_j:\D^N\to \D$ be the projection on the $j$-th component. Now,  for every $k$, there exists $j_k\in \{1,\ldots, N\}$ such that 
\[
K_{\D^N}(\underline{0}, F^{\circ n_k}(\underline{0}))=K_{\D}(\pi_{j_k}(\underline{0}), \pi_{j_k}(F^{\circ n_k}(\underline{0})).
\]
Since $\{n_k\}$ converges to $+\infty$ and  $j_k\in \{1,\ldots, N\}$ (a finite set), there exists $j_0\in \{1,\ldots, N\}$ such that $j_k=j_0$ for infinitely many $k$. Up to extracting subsequences we can thus assume that $K_{\D^N}(\underline{0}, F^{\circ n_k}(\underline{0}))=K_{\D}(\pi_{j_0}(\underline{0}), \pi_{j_0}(F^{\circ n_k}(\underline{0}))$ for all $k$. Also, up to permuting coordinates, we can assume that $j_0=1$, that is, for all $k\in \mathbb N$,
\begin{equation}\label{Eq:good-choice}
K_{\D^N}(\underline{0}, F^{\circ n_k}(\underline{0}))=K_{\D}(\pi_{1}(\underline{0}), \pi_{1}(F^{\circ n_k}(\underline{0}))).
\end{equation}
Note that $\{\pi_{1}(F^{\circ n_k}(\underline{0}))\}$ converges to a point in $\partial \D$ (since $\{K_{\D^N}(\underline{0}, F^{\circ n_k}(\underline{0}))\}$---and hence $\{K_{\D}(0, \pi_{1}(F^{\circ n_k}(\underline{0}))\}$---converges to $+\infty$) and $\{F^{\circ n_k}(\underline{0})\}$ converges to a point in $\partial (\D^N)$. Thus, $\{F^{\circ n_k}(\underline{0})\}$ converges to $(q_1,q_2,\ldots, q_N)$, with $|q_1|=1$. 

Let $\nu$ be the number of $q_j$'s such that $|q_j|=1$, $1\leq \nu\leq N$. Up to  permutation of coordinates we can assume that $|q_j|=1$ for $j=1,\ldots, \nu$ and $|q_j|<1$ for $j=\nu+1,\ldots, N$. Then, replacing $F$ with the $T^{-1} \circ F \circ T$, where $T$ is the automorphism of $\D^N$ given by 
\[
(z_1,\ldots, z_N)\mapsto \left(q_1z_1, \ldots,  q_\nu z_\nu, z_{\nu+1}, \ldots, z_{N}\right),
\]
we can assume that $\{F^{\circ n_k}(\underline{0})\}$ converges to $q:=(1,\ldots, 1, \alpha_{\nu+1}, \ldots, \alpha_N)$ where $|\alpha_j|<1$, $j=\nu+1,\ldots, N$.

Let $w_k:=F^{\circ n_k}(\underline{0})$. Note that, by \eqref{Eq:good-choice}, we have 
\begin{equation}\label{Eq:equal-up-down}
K_{\D^N}(\underline{0},w_k)=K_\D(0,\pi_1(w_k)),
\end{equation}
 for all $k\in \mathbb N$.  Fix $m\in \mathbb N$.  Let 
\[
w_{k}^m:=F^{\circ (n_k-m)}(\underline{0}).
\]
Up to extracting subsequences if necessary, we can assume that $\{w_{k}^m\}$ converges to some point $(q_1^m, q_2^m,\ldots, q_N^m)\in\partial (\D^N)$. Note that for $j=1,\ldots, N$
\[
K_\D(\pi_j(w_k), \pi_j(w_{k}^m))\leq K_{\D^N}(w_k, w_{k}^m)\leq K_{\D^N}(F^{\circ m}(\underline{0}), \underline{0})<+\infty.
\]
This implies that $\{\pi_j(w_{k}^m)\}$ converges to $1$ for $j=1,\ldots, \nu$ and$\{\pi_j(w_{k}^m)\}$ converges to some point in $\D$, $j=\nu+1,\ldots, N$,  that is, $\{w_{k}^m\}$ converges to some point $q_m=(1,\ldots, 1, \alpha_{\nu+1}^m, \ldots, \alpha_N^m)$ for some $\alpha_j^m\in \D$ as $k\to\infty$.
 Also,
\begin{equation}\label{Eq:JWC1}
\lim_{k\to \infty}\pi_1(F^{\circ m}(w_{k}^m))=\lim_{k\to \infty}\pi_1(F^{\circ m}(F^{\circ (n_k-m)}(\underline{0}))=\lim_{k\to \infty}\pi_1(F^{\circ n_k}(\underline{0}))=1.
\end{equation}

Now, for every $n_k>m$, 
\begin{equation}\label{Eq:JWC2}
\begin{split}
 &K_{\D^N}(\underline{0}, w_k^m)-K_\D(0, \pi_1(F^{\circ m}(w_k^m)))\\&=K_{\D^N}(\underline{0}, F^{\circ (n_k-m)}(\underline{0}))-K_\D(0, \pi_1(F^{\circ m}(F^{\circ (n_k-m)}(\underline{0}))))\\ &\leq
K_{\D^N}(\underline{0}, F^{\circ n_k}(\underline{0}))-K_\D(0, \pi_1(w_k))\\&=K_{\D^N}(\underline{0}, w_k)-K_\D(0, \pi_1(w_k))\stackrel{\eqref{Eq:equal-up-down}}{=}0.
\end{split}
\end{equation}

Let $f:=\pi_1\circ F^{\circ m}:\D^N\to \D$. Equations \eqref{Eq:JWC1} and \eqref{Eq:JWC2}, with $\xi_k=w_{k-m}$, allows us to apply Theorem~\ref{abatep} to such an $f$. Hence,
\[
\limsup_{r\to 1^-}\frac{|1-\pi_1(F^{\circ m}(rq_m))|}{1-r}\leq 1.
\]
This implies that 
\begin{equation*}
\begin{split}
&\liminf_{\D\ni \zeta\to 1}\frac{1-|\pi_1(F^{\circ m}(\zeta q_m))|}{1-|\zeta|}\leq 
\limsup_{r\to 1^-}\frac{1-|\pi_1(F^{\circ m}(rq_m))|}{1-r}\\&
\leq \limsup_{r\to 1^-}\frac{|1-\pi_1(F^{\circ m}(rq_m))|}{1-r}\leq 1.
\end{split}
\end{equation*}

The classical Julia's lemma (see, {\sl e.g.}, \cite[Theorem~1.4.7]{BCDM}) implies then \eqref{classJulia}---with $j_0=1$. 
Equation \eqref{classJulia} implies that, if $\{\pi_1(F^{\circ m_k}(\underline{0})\}$ converges to some point in $\partial\D$, then actually   $\{\pi_1(F^{\circ m_k}(\underline{0})\}$ converges to $1$.

Now, let $\{m_k\}$ be a sequence such that $\{F^{\circ m_k}\}$ converges uniformly on compacta and $q_{m_k}$ converges to some $q_\infty=(1,\ldots, 1, \beta_{\nu+1},\ldots, \beta_N)$, where $\beta_j\in\overline{\D}$, $j=\nu+1,\ldots, N$. Thus, $\{\pi_1(F^{\circ m_k}(\zeta q_m)\}$ converges uniformly on compacta of $\D$ to a holomorphic map $g:\D\to \D\cup\{1\}$. 

Assume $g(\D)\subset \D$. By \eqref{classJulia}, for all $\zeta\in \D$ we have
\[
\frac{|1-g(\zeta) |^2}{1-|g(\zeta)|^2}\leq \frac{|1-\zeta |^2}{1-|\zeta|^2},
\]
thus $g$ is not constant, and it is either the identity or a holomorphic self-map of $\D$ with no fixed points and  $1$ is its Denjoy-Wolff point.
\end{proof}

Note that \eqref{classJulia} and Lemma~\ref{Lem:orbits} imply that if 
\[
\liminf_{m\to\infty} k_{\D^N}(\underline{0}, \pi_{j_0}(F^{\circ m}(\underline{0})))=+\infty,
\]
then $\{\pi_{j_0}(F^{\circ m}(z)\}$ converges to $\sigma$ for all $z\in \D^N$ as $m\to\infty$.

In general, we can ask the following question:
\begin{question}
Assume that 
\[
\Omega=D_1\times \ldots \times D_\ell,
\]
where $D_j$ is a bounded convex domain in $\C^{N_j}$, $j=1,\ldots, \ell$ and $N_1+\ldots+N_\ell=N\geq 2$. Suppose $D_j$ has the Denjoy-Wolff property for every $j=1,\ldots, \ell$. Let $F:\Omega\to\Omega$ be holomorphic and such that $F$ has no fixed points in $\Omega$. Is it true that there exist $j_0\in\{1,\ldots, \ell\}$ and $p_{j_0}\in \partial D_{j_0}$ such that $\{\pi_{j_0}(F^{\circ m}(z))\}$ converges to $p_{j_0}$ for all $z\in \Omega$ as $m\to \infty$? 
\end{question}

\end{document}